\documentclass[11pt]{amsart}

\usepackage{a4}
\usepackage{amsfonts}
\usepackage{amsmath}
\usepackage{amsthm}
\usepackage{url}
\usepackage[all]{xy}
\usepackage{graphicx}
\usepackage{latexsym}
\usepackage{amssymb}
\usepackage[cp850]{inputenc}
\usepackage{epsfig}
\usepackage{psfrag}
\usepackage{amsfonts}
\usepackage{dsfont}

\newtheorem{sat}{Theorem}[section]		\newtheorem{lem}[sat]{Lemma}
			\newtheorem{prop}[sat]{Proposition}
				
\newtheorem*{defi*}{Definition}			\newtheorem*{bei*}{Example}
\newtheorem*{sat*}{Theorem}				\newtheorem*{kor*}{Corollary}
\newtheorem*{rmk*}{Remark}					
\newtheorem*{quest*}{Question}	

\let\ssection=\section
\renewcommand{\section}{\setcounter{equation}{0}\ssection}

\newtheorem*{namedtheorem}{\theoremname}
\newcommand{\theoremname}{testing}
\newenvironment{named}[1]{\renewcommand{\theoremname}{#1}\begin{namedtheorem}}{\end{namedtheorem}}

\theoremstyle{remark}
\newtheorem*{bem}{Remark}

\newcommand{\BC}{\mathbb C}			\newcommand{\BH}{\mathbb H}
\newcommand{\BR}{\mathbb R}			
\newcommand{\BN}{\mathbb N}			
\newcommand{\BS}{\mathbb S}			\newcommand{\BZ}{\mathbb Z}

		\newcommand{\CB}{\mathcal B}
		\newcommand{\calD}{\mathcal D}

		\newcommand{\CN}{\mathcal N}

\newcommand{\CS}{\mathcal S}

\newcommand{\D}{\partial}

\DeclareMathOperator{\PSL}{PSL}		
\DeclareMathOperator{\Isom}{Isom}	

\DeclareMathOperator{\diam}{diam}

\newcommand{\comment}[1]{}

\begin{document}

\title[]{A Cantor set with hyperbolic complement}
\author{Juan Souto and Matthew Stover}
\thanks{Juan Souto was partially supported by NSERC Discovery and Accelerator Supplement grants. Matthew Stover was partially supported by NSF RTG grant DMS 0602191}
\dedicatory{To Dick Canary on the occasion of his $50^{th}$ birthday}

\begin{abstract}
We construct a Cantor set in $\BS^3$ whose complement admits a complete hyperbolic metric.
\end{abstract}

\maketitle

\section{Introduction}

Recall that a {\em Cantor set} is a metrizable compactum which is totally disconnected and has no isolated points. While any two Cantor sets are homeomorphic to each other, it is well-known that there are Cantor sets embedded in Euclidean space such that no homeomorphism between them extends to an ambient self-homeomorphism. The first example of this phenomenon is due to Antoine \cite{Antoine} who constructed a Cantor set in $\BR^3$ whose complement is not simply connected, and hence not homeomorphic to the complement of the standard dyadic Cantor set. Following Antoine's work there has been a small industry devoted to constructing examples of {\em wild} Cantor sets in $\BR^3$, or more generally $\BR^n$, having various pathological properties; see for example \cite{Bestvina-Cooper,Bing,DO,FS,Sher,Skora} and the references therein. In this note we construct yet another example of a wild Cantor set:

\begin{sat}\label{sat1}
There is a Cantor set $C\subset\BS^3$ whose complement $\BS^3\setminus C$ admits a complete hyperbolic metric.
\end{sat}

To construct the Cantor set $C$ provided by Theorem \ref{sat1} we will mimic the construction of Antoine's necklace, using knotted and linked $\theta$-graphs instead of circles. The bulk of the work is to give a sufficient condition for an open manifold with infinite topology to admit a complete hyperbolic metric. More concretely, we prove that an open $3$-manifold is hyperbolic if it admits a nested exhaustion $M=\bigcup_nK_n$ such that the closure of $K_n\setminus K_{n-1}$ is acylindrical for all $n$ and such that the genus of each component of $\D K_n$ is bounded by some constant independent of $n$. This result is not going to surprise any expert on Kleinian groups, and the proof uses rather standard arguments.
\medskip

In this note we just prove Theorem \ref{sat1} as stated above. However, an argument that is painful but relatively straightforward for experts shows that the Cantor set $C$ can also be constructed so that {\em any two orientation preserving embeddings of $\BS^3\setminus C$ into $\BS^3$ are isotopic to each other} (compare with \cite{balkanic,Kent-Souto}). One can also construct $C$ in such a way that {\em the hyperbolic metric on $\BS^2\setminus C$ is unique up to isometry}. 
In fact, the following is, at least from the point of view of Kleinian groups, an interesting question:

\begin{quest*}
Is there a Cantor set $C\subset\BS^3$ whose complement admits non-isometric complete hyperbolic metrics?
\end{quest*}

Also, note that, as is the case with Antoine's necklace \cite{Sher}, 
one can use Theorem \ref{sat1} to construct uncountably many homeomorphism classes of Cantor set complements in $\BS^3$ admitting a complete hyperbolic metric.
\medskip

This paper is organized as follows. In section \ref{sec-top} we discuss a few facts from $3$-dimensional topology used later on. In section \ref{sec-geom} we show that $3$-manifolds that admit what we call a {\em nested exhaustion with truly excellent gaps} are hyperbolic. Theorem \ref{sat1} is proved in section \ref{sec-cantor}.
\medskip

\noindent{\bf Acknowledgements.} The first author would like to thank Mario Bonk and Vlad Markovic for many amusing conversations on Cantor sets and other topics. The images in this paper were created with Google Sketchup.

\section{}\label{sec-top}
We refer to \cite{Jaco} and \cite{Marden} for basic facts from $3$-manifold topology and hyperbolic geometry.

Following Myers \cite{Myers}, we say that a compact orientable $3$-manifold $M$ is {\em excellent} if it is irreducible, atoroidal, and acylindrical. An excellent $3$-manifold all of whose boundary components have negative Euler characteristic is {\em truly excellent}. Suppose that $M$ is a compact orientable $3$-manifold whose boundary $\D M$ does not contain $2$-spheres. It follows from Perelman's proof of the Poincare conjecture that $M$ is truly excellent if and only if its fundamental group $\pi_1(M)$ is infinite, does not contain $\BZ^2$ as a subgroup, and splits neither over the trivial group nor over $\BZ$. Yet another characterization, due to Thurston in the presence of boundary and to Perelman in general, is that a compact manifold $M$ is truly excellent if and only if it admits a hyperbolic metric with totally geodesic boundary.

\begin{bem}
Notice that it follows from the observations above that every compact $3$-manifold $M$ which is homotopy equivalent to a (truly) excellent manifold $M'$ is (truly) excellent as well. In particular, it follows from the work of Johannson \cite{Johannson} that $M$ and $M'$ are actually homeomorphic.
\end{bem}

Later on we will need to use over and over again that appropriately glued truly excellent $3$-manifolds yield again a truly excellent manifold. Before stating what we will need in a lemma, recall that a subgroup $H$ of a group $G$ is {\em malnormal} if 
$$\{g\in G\ \vert\ gHg^{-1}\cap H\neq 1_G\} = H.$$
In terms of covering theory this translates to the following fact: if $X$ is a simply connected space on which $G$ acts freely and discretely, and if $\gamma,\gamma'\subset X/H$ are homotopically essential closed curves whose images under the covering map $\pi:X/H\to X/G$ are equal, $\pi(\gamma)=\pi(\gamma')$, then $\gamma=\gamma'$.



\begin{lem}\label{gluing}
Suppose that $N$ is a compact oriented $3$-manifold with boundary $\D N$ and that $S\subset\D N$ a disconnected subsurface of the boundary. Let $\tau:S\to S$ be an orientation reversing involution that preserves no connected component of $S$. Finally, consider the oriented manifold $M=N/\tau$ obtained by gluing $N$ according to $\tau$ and suppose that
\begin{itemize}
\item $N$ is truly excellent,
\item each component of $S$ has negative Euler characteristic and is $\pi_1$-injective in $\D N$, and
\item $\D M$ does not contain tori.
\end{itemize}
Then $M$ is truly excellent, and moreover:
\begin{enumerate}
\item The image of every component $\Sigma$ of $S$ under the inclusion $\pi:N\to M$ is a $2$-sided incompressible surface.
\item If $\Sigma$ (resp.~$U$) is a connected component of $S$ (resp.~$N$) then the subgroup $\pi_*(\pi_1(S))$ (resp.~$\pi_*(\pi_1(U))$) is malnormal in $\pi_1(M)$. 
\end{enumerate}
\end{lem}

Lemma \ref{gluing} follows either easily from standard innermost arguments (see \cite[Section 2]{Myers}) or from well-known results from Bass--Serre theory on amalgamating groups along a common malnormal subgroup. We leave the details to the reader.

%
%
%
%
%

\begin{lem}\label{lem1}
Let $M$ be a complete open hyperbolic $3$-manifold, $K$ a truly excellent compact $3$-manifold, and $\iota:K\to M$ a homotopy equivalence. If the restriction of $\iota$ to $\D K$ is an embedding, then $\iota$ is homotopic relative to $\D K$ to an embedding $\iota':K\to M$.
\end{lem}
\begin{proof}
Since $K$ is compact and $M$ is homotopy equivalent to $K$, $\pi_1(M)$ is finitely generated. 
In particular, $M$ is homeomorphic to the interior of a compact manifold $\bar M$ \cite{Agol,CG}, and we can assume that $\iota(K)$ is contained in the interior of $\bar M$. The homotopy equivalence $\iota:K\to\bar M$ is homotopic to a homeomorphism $\tau:K\to\bar M$ \cite{Johannson}. Now, let $S\subset\D K$ be a boundary component of $K$. Since $\iota(S)$ and $\tau(S)$ are disjoint, $\pi_1$-injective, and homotopic, it follows from Waldhausen's cobordism theorem \cite{Waldhausen} that there is $U_S\subset\bar M$ homeomorphic to $S\times[0,1]$ with $\D U_S=\iota(S)\cup\tau(S)$.

Notice that if $S'\subset\D K$ is another boundary component of $K$ then we have $\iota(S')\cap\D U_S=\emptyset$, meaning that either $\iota(S')\subset U_S$ or $\iota(S')\cap U_S=\emptyset$. We rule out the former possibility: If $\iota(S')\subset U_S$ then it is a closed embedded $\pi_1$-injective surface in the trivial interval bundle $U_S$ and hence is isotopic to the boundary components of $U_S$. Since $\iota$ is a homotopy equivalence, this implies that $S$ and $S'$ are homotopic in $K$, but this contradicts the assumption that $K$ is truly excellent. This proves that $\iota(S')\cap U_S=\emptyset$ for all $S'\subset\D K\setminus S$. Notice that the same argument shows that $U_S\cap U_{S'}=\emptyset$ for all distinct boundary components $S,S'\subset\D K$ of $K$. 

Finally, let $\hat M$ be the submanifold of 
$\bar M$ obtained by removing $U_S \setminus \iota(S)$ for all $S\subset\D K$, and notice that the homotopy equivalence $\iota:K\to M$ is homotopic relative to $\D K$ to a homotopy equivalence $\hat\iota:K\to\hat M$ whose restriction to $\D K$ is a homeomorphism onto $\D\hat M$. The homotopy equivalence $\hat\iota:K\to\hat M$ is homotopic relative to $\D K$ to a homeomorphism $\iota':K\to\hat M\subset M$ \cite{Waldhausen}. This is our desired embedding.
\end{proof}

\begin{lem}\label{lem2}
Let $M$ and $M'$ be hyperbolic $3$-manifolds, $\pi:M\to M'$ be a covering, $K\subset M$ a compact core, and assume that $\pi_*(\pi_1(K))$ is malnormal in $\pi_1(M')$. If the restriction of $\pi$ to $\D K$ is an embedding, then so is the restriction of $\pi$ to $K$.
\end{lem}

Recall that a compact core of a $3$-manifold $M$ is a compact submanifold $K$ such that the inclusion $K\hookrightarrow M$ is a homotopy equivalence, and observe that every manifold admitting a compact core has finitely generated fundamental group. We also note that in the statement of Lemma \ref{lem2} we do not assume that $\pi_1(M')$ is finitely generated.

\begin{proof}
Notice that it suffices to prove that if $x\in\D K$ and $y\in K$ are two points with $\pi(x)=\pi(y)$, then $x=y$. Let $S\subset\D K$ be the connected component containing $x$, consider its image $\pi(S)$, and let $S'$ be the component of $\pi^{-1}(\pi(S))$ with $y\in S'$. Notice that $S'\subset K$. In particular, compactness of $K$ implies that $S'$ is compact and hence that the cover $\pi\vert_{S'}:S'\to S$ is finite-to-one. 

It follows that there are two curves $\gamma\subset S$ and $\gamma'\subset S'$ which are essential in $\pi_1(M)$ such that $\pi(\gamma)=\pi(\gamma')$. The condition that $\pi_*(\pi_1(K))$ is a malnormal subgroup of $\pi_1(M')$ implies that $\gamma=\gamma'$ and hence that $S'\cap S\neq\emptyset$. Since both $S$ and $S'$ are components of the preimage of the embedded surface $\pi(S)$, it follows that $S'=S$ and thus that $y\in S$. Injectivity of $\pi$ on $S$ implies that $\pi(x)=\pi(y)$, as we needed to prove.
\end{proof}

\section{}\label{sec-geom}
In this note we will be interested in $3$-manifolds obtained by gluing truly excellent manifolds along their boundaries. More concretely we consider open manifolds $M$ which admit a {\em nested exhaustion with truly excellent gaps}
$$M=\bigcup_{n\in\BN} K_n,\ \ K_0=\emptyset,$$
by which we mean that $K_n$ is contained in the interior of $K_{n+1}$ and that the closure of $K_{n+1}\setminus K_n$ is a truly excellent manifold for all $n$. We prove:

\begin{prop}\label{kor-metric}
Every open $3$-manifold which admits a nested exhaustion with truly excellent gaps 
$$M=\bigcup_{n\in\BN} K_n,\ \ K_0=\emptyset$$
such that there is an upper bound for the genus of the connected components of $\bigcup_n\D K_n$ is homeomorphic to a complete hyperbolic $3$-manifold.
\end{prop}

Suppose throughout this section that $M$ and $K_n$ are as in the statement of Proposition \ref{kor-metric} and fix $*\in K_1$. Before launching into the proof of the proposition, observe that Lemma \ref{gluing} implies:
\begin{enumerate}
\item $K_n$ is truly excellent for all $n\ge 1$.
\item If $S\subset\D K_n$ is a connected component of the boundary of $K_n$, then $S$ is incompressible in $M$.
\item If $S\subset\D K_n$ and $S'\subset\D K_m$ are connected components of the boundary of $K_n$ and $K_m$ with $n,m\ge 1$ such that there are essential curves $\gamma\subset S$ and $\gamma'\subset S'$ which are freely homotopic in $M$ then $n=m$, $S=S'$ and $\gamma$ and $\gamma'$ are in fact freely homotopic within $S$.
\item $\pi_1(M,*)$ is the nested union of the subgroups $\pi_1(K_n,*)$.
\item $\pi_1(K_n,*)$ is malnormal in $\pi_1(M,*)$ for all $n\ge 1$.
\end{enumerate}
Also, notice that since $M$ is the nested union of aspherical manifolds, it is aspherical as well.

We divide the proof of Proposition \ref{kor-metric} into two separate statements:

\begin{lem}\label{hom-metric}
The manifold $M$ is homotopy equivalent to a complete hyperbolic $3$-manifold $M'$.
\end{lem}

\begin{proof}
Since $K_n$ is excellent for all $n\ge 1$, there is a discrete and faithful representation 
$$\rho_n:\pi_1(K_n,*)\to\PSL_2\BC=\Isom_+(\BH^3).$$
See \cite{Kapovich,Ota98}. Fixing $k$, for each for $n\ge k$ we can restrict the representation $\rho_n$ to the subgroup $\pi_1(K_k,*)$ of $\pi_1(K_n,*)$. Since each of the manifolds $K_k$ is excellent, it follows from Thurston's compactness theorem \cite{Thurston} that there is a sequence $(g_n)\subset\PSL_2\BC$ such that for all $\gamma\in\pi_1(K_k,*)$ the sequence $(g_n\rho_n(\gamma)g_n^{-1})$ is relatively compact in $\PSL_2\BC$. In particular, conjugating our representations and passing to a diagonal subsequence we can assume that the limit 
$$\rho(\gamma)=\lim_{n\to\infty,\ n\ge k}\rho_n(\gamma)$$ 
exists for all $\gamma\in\pi_1(K_k,*)$ and for all $k$. Since $\pi_1(M,*)=\bigcup_k\pi_1(K_k,*)$ we therefore obtain a representation 
$$\rho:\pi_1(M)\to\PSL_2\BC.$$
It is discrete and faithful by work of J\o rgensen \cite{jorgmarden}. In particular, $M'=\BH^3/\rho(\pi_1(M,*))$ is a hyperbolic $3$-manifold with $\pi_1(M')\simeq\pi_1(M)$. Since both $M$ and $M'$ are aspherical, it follows that they are homotopy equivalent.
\end{proof}

We now prove that $M$ and the manifold $M'$ provided by Lemma \ref{hom-metric} are not only homotopy equivalent but actually homeomorphic:

\begin{lem}\label{homeo-metric}
$M$ and $M'$ are homeomorphic.
\end{lem}

\begin{proof}
Choose a homotopy equivalence $\phi:M\to M'$ and set 
$$\CS=\bigcup_n\D K_n\subset M$$
Every connected component $S$ of $\CS$ is $\pi_1$-injective; it follows that $\phi(S)$ is homotopic to an immersed least area surface in $M'$ \cite{SY}. We can thus assume, up to replacing $\phi$ by a homotopic map, that the restriction of $\phi$ to $S$ is a minimal immersion for all components $S$ of $\CS$.

We now claim that the restriction of $\phi$ to $\CS$ is proper. In fact, if that were not the case then there would be a sequence $(S_k)$ of distinct components of $\CS$ and a sequence of points $p_k\in S_k$ such that $\phi(p_k)$ has a limit in $M'$. Endow $S_k$ with the pulled-back Riemannian metric and notice that, since $\phi\vert_{S_k}$ is a minimal immersion, this metric has curvature bounded from above by $-1$. Since, by assumption, $S_k$ has genus uniformly bounded from above, there is $C$ such that $S_k$ has at most area $C$ for all $k$. Therefore, there is some constant $D>0$ such that for all $k$ there is a homotopically essential loop $\gamma_k\subset S_k$ based at $p_k$ whose image $\phi(\gamma_k)$ has at most length $D$. Since the points $\phi(p_k)$ converge in $M'$, it follows that the loops $\phi(\gamma_k)$ belong to finitely many free homotopy classes in $M'$. In other words, there are $l\neq k$ such that $\phi(\gamma_k)$ and $\phi(\gamma_l)$ represent the same conjugacy class in $\pi_1(M')$. Since $\phi$ is a homotopy equivalence, it follows that $\gamma_k$ and $\gamma_l$ also represent the same conjugacy class in $\pi_1(M)$; as we noted above this is not possible. This shows that the restriction of $\phi$ to $\CS$ is proper.

We next prove that the restriction of $\phi$ to $\CS$ is an embedding. Properness implies that it suffices to show that the restriction of $\phi$ to $\CS_n=\bigcup_{i\le n} \D K_i$ is an embedding for all $n$. Notice that there is some $N\ge n$ such that a neighborhood of the set  $\phi(K_n)$ lifts homeomorphically under the cover $\pi:\BH^3/\rho(\pi_1(K_N))\to M'$. Now consider the commutative diagram
$$\xymatrix{& \BH^3/\rho(\pi_1(K_N)) \ar[d]^{\pi}\\
K_N\ar[ru]^{\tilde\phi}\ar[r]_{\phi} & M'}$$
The manifold $\BH^3/\rho(\pi_1(K_N))$ is homeomorphic to the interior of a compact manifold $\bar M_N$ by the tameness theorem \cite{Agol,CG}. Moreover, since $K_N$ is excellent it follows from \cite{Johannson} that the map $\tilde\phi:K_N\to\bar M_N$ is homotopic to a homeomorphism. In particular, this implies that $\tilde\phi(\CS_n)$ is homotopic to an embedded surface. Since $\tilde\phi(\CS_n)$ is a $\pi_1$-injective least area surface and since two curves in $\CS_n$ which are homotopic in $M$ are also homotopic within $\CS_n$, it follows from \cite{FHS} that the restriction of $\tilde\phi$ to $\CS_n$ is an embedding. Since the restriction of $\pi$ to $\tilde\phi(K_n)$ is a homeomorphism onto $\phi(K_n)$, it follows that the restriction of $\phi$ to $\CS_n$ is also an embedding. This proves that $\phi$ maps $\CS$ homeomorphically onto its image.

At this point we are ready to finish the proof. It suffices to prove that for every natural number $n$, the restriction of $\phi$ to the closure $U_n$ of $K_n\setminus K_{n-1}$ is homotopic rel $\D U_n\subset\CS$ to an embedding. Consider the cover $\pi':\BH^3/\rho(\pi_1(U_n))\to M'$ and observe that we have a diagram as follows:
$$\xymatrix{
& \BH^3/\rho(\pi_1(U_n)) \ar[d]^{\pi'}\\
U_n\ar[ru]^{\tilde\phi}\ar[r]_{\phi} & M'}$$
Since the restriction of $\phi$ to $\D U_n\subset\CS$ is an embedding, we deduce that the restriction of $\tilde\phi$ to each $\D U_n$ is also an embedding. It follows from Lemma \ref{lem1} that $\tilde\phi$ is homotopic relative to $\D U_n$ to an embedding $\tilde\psi:U_n\to\BH^3/\rho(\pi_1(U_n))$. Since $\tilde\psi(U_n)$ is a compact core of $\BH^3/\rho(\pi_1(U_n))$, Lemma \ref{lem2} implies that that $\pi\circ\tilde\psi$ is also an embedding. 
\end{proof}

Proposition \ref{kor-metric} follows directly from Lemma \ref{hom-metric} and Lemma \ref{homeo-metric}.\qed

\section{}\label{sec-cantor}
In this section we prove Theorem \ref{sat1}. The basic idea is to modify the construction of Antoine's necklace by replacing each link of the necklace by a graph. A (piecewise linearly) embedded finite graph $X$ in a manifold $M$ with possibly non-empty boundary is {\em properly embedded} if $X\cap\D M$ is precisely equal to the set of vertices of $X$ with valence $1$. If $X\subset M$ is any such properly embedded graph then we denote by $\CN(X)$ an open regular neighborhood of $X$. A properly embedded graph $X\subset M$ is {\em truly excellent} if $M\setminus\CN(X)$ is truly excellent.

In the proof of the Theorem \ref{sat1} we will make heavy use of the following result which is basically due to Myers \cite{Myers} (see also \cite{Kent-Souto}):

\begin{sat*}[Myers]
Let $M$ be an oriented $3$-manifold and $X\subset M$ a properly embedded finite graph such that every component of $\D(M\setminus\CN(X))$ has negative Euler characteristic. Then $X$ is homotopic, relative to $\D M$, to a truly excellent properly embedded graph.
\end{sat*}

After these remarks we are ready to prove Theorem \ref{sat1}:

\begin{named}{Theorem \ref{sat1}}
There is a Cantor set $C\subset\BS^3$ whose complement $\BS^3\setminus C$ admits a complete hyperbolic metric.
\end{named}
\begin{proof}
We will construct a sequence $V_0,V_1,V_2,\dots$ of compact $3$-manifolds in $\BS^3$ starting with $V_0=\BS^3$ and satisfying the following conditions for all $n\ge 1$:
\begin{enumerate}
\item Each $V_n$ is contained in the interior of $V_{n-1}$ and the closure of $V_{n-1}\setminus V_n$ is truly excellent.
\item Each component of $\D V_n$ has genus $2$.
\item Each component of $V_{n-1}$ contains at least two components of $V_n$.
\item For any sequence of components $U_n\subset V_n$, $\lim_{n\to\infty}\diam(U_n)=0$, where $\diam(\cdot)$ is the diameter in the spherical metric on $\BS^3$.
\end{enumerate}
Assuming for a moment that such a sequence exists, we conclude the proof of Theorem \ref{sat1}. The set $C=\bigcap_{n=1}^\infty V_n$ is an intersection of compact sets and hence compact. Moreover, (3) implies that it is totally disconnected and (2) yields in turn that $C$ has no isolated points. In other words, $C$ is a Cantor set. Now let $K_n$ be the closure of $\BS^3\setminus V_n$ for $n\ge 1$ and notice that
$$K_1\subset K_2\subset K_3\subset\cdots\ \ \hbox{and}\ \ M=\BS^3\setminus C=\bigcup_n K_n.$$
Proposition \ref{kor-metric} applies by (1), which implies that $M=\BS^3\setminus C$ admits a complete hyperbolic metric, as we wanted to show.

It remains to construct the submanifolds $V_n$ of $\BS^3$ satisfying (1)-(4). We will proceed by induction, constructing $V_n$ as a regular neighborhood of linked $\theta$-graphs; by a $\theta$-graph, we mean a trivalent graph with $2$ vertices and without separating edges. To begin, let $X,X'\subset\BS^3$ be disjoint embedded $\theta$-graphs whose union $X\cup X'$ is a truly excellent graph, and let $V_1$ be a regular neighborhood of $X\cup X'$ in $V_0=\BS^3$. The $\theta$-graphs $X$ and $X'$ exist by Myers's Theorem.

Suppose that we constructed $V_{n-1}$ and let $U\subset V_{n-1}$ be one of its connected components. By induction, $U$ is a genus 2 handlebody. We are going to construct a disconnected graph $L_U$ contained in the interior of $U$, such that each one of its connected components is a $\theta$-graph of at most diameter $2^{-n}$ and such that $U\setminus\CN(L_U)$ is truly excellent. Once this graph $L_U$ exists, we define $V_n$ as the union of the submanifolds $\CN(L_U)\subset U$ over all connected components $U$ of $V_{n-1}$. 

It remains to construct the link $L=L_U$ in the genus 2 handlebody $U$. To start, let $X\subset U$ be a spine of $U$, i.e., a $\theta$-graph whose complement is a product. We now take a very slim regular neighborhood $W$ of $X$ constructed out of closed topological balls
\begin{equation}\label{eq-balls}
W=A_1\cup A_2\cup B_1^1\cup\dots\cup B_{k_1}^1\cup B_1^2\cup\dots\cup B_{k_2}^2 \cup B_1^3\cup\dots\cup B_{k_3}^3
\end{equation}
satisfying:
\begin{itemize}
\item The interiors of all the pieces $A_i,B_l^j$ are disjoint.
\item The intersections $A_1\cap B_1^j$ and $A_2\cap B_{k_j}^j$ are 2-dimensional disks for $j=1,2,3$.
\item $B_i^j\cap B_{i+1}^j$ is a 2-dimensional disk for all $j=1,2,3$ and $i=1,\dots,k_j-1$.
\item All other intersections are trivial.
\end{itemize}
The reader should think of $A_1,A_2$ as regular neighborhoods of the two vertices of $X$ and for $j=1,2,3$ think of $A_1\cup B_1^j\cup\dots\cup B_{k_j}^j\cup A_2$ as a regular neighborhood of an edge of $X$; compare with Figure \ref{fig1}.

\begin{figure}[h]
        \centering
         \includegraphics[width=10cm]{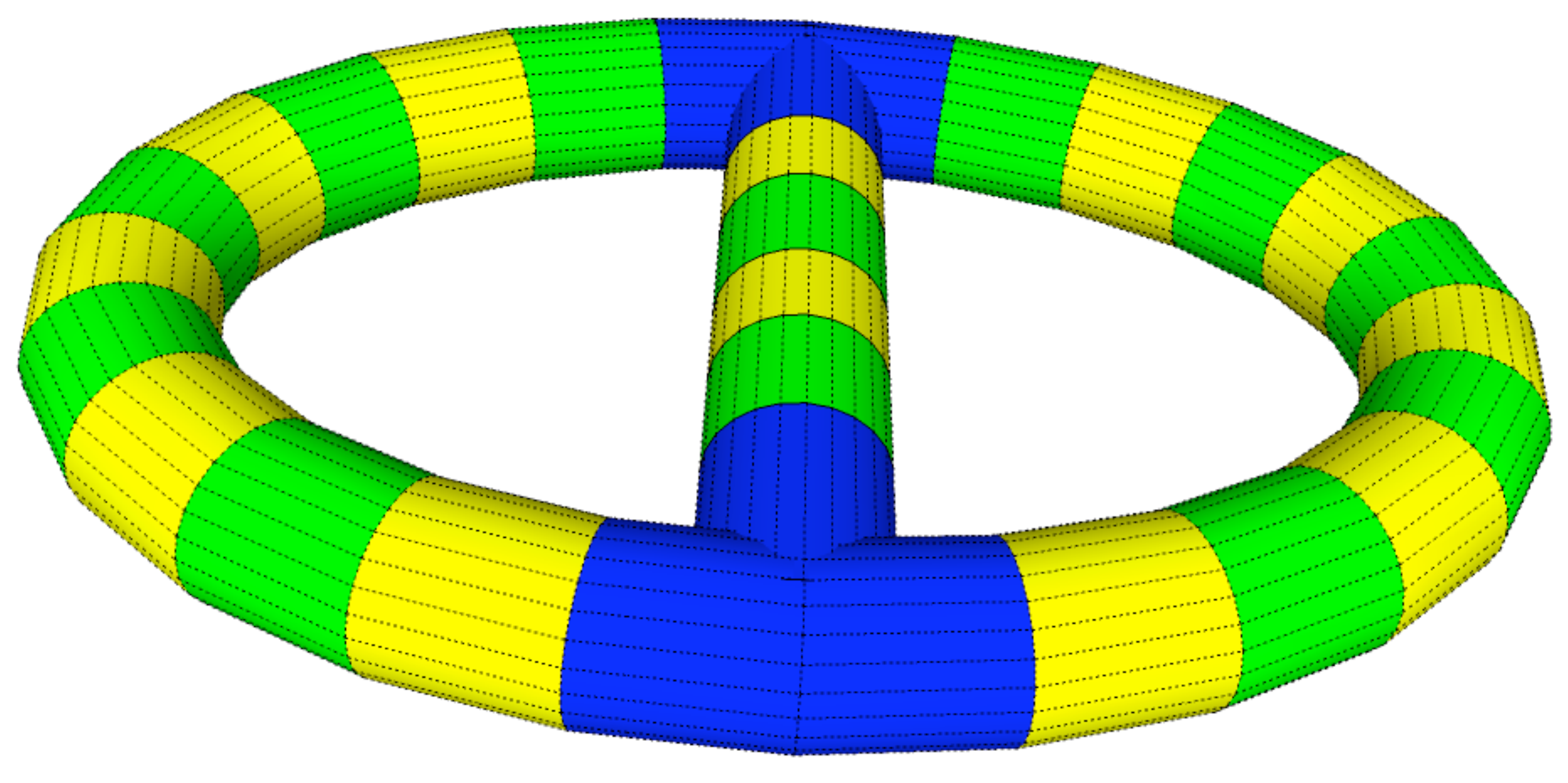}
         \caption{The neighborhood $W$ of a $\theta$-graph cut into the pieces $A_1, A_2$ (blue) and $B_k^j$ (green and yellow).}\label{fig1}
\end{figure}

We can clearly choose $W$ in \eqref{eq-balls} in such a way that each one of the pieces $A_i,B_l^j$ has diameter at most $2^{-(n+1)}$. We construct $L$ in such a way that each component is contained in the union of two of the pieces of \eqref{eq-balls}. This yields the desired diameter bound.

Denote by $\CB$ the collection of pieces $A_i,B_l^j$ over all choices of $i,j,l$, let $\calD$ be the union of the intersections $E\cap F$ with $E,F\in\CB$, and for every component $D$ of $\calD$ choose three distinct points $p^D_1,p^D_2,p^D_3$ in its interior. Notice that if $D$ is a component of $\calD$ contained in the boundary $\D E$ of some $E\in\CB$, then $D\setminus\{p^D_1,p^D_2,p^D_3\}$ is a $\pi_1$-injective subsurface of negative Euler characteristic of $\D E\setminus\bigcup_{D\in\calD, D\subset E}\{p^D_1,p^D_2,p^D_3\}$. Continuing with the same notation, for each component $D$ of $\calD$ contained in $\D E$ choose a properly embedded tripod $T(E,D)$ with endpoints in the punctures of $D$ and assume that $T(E,D)\cap T(E,D')=\emptyset$ for any two components $D,D'$ of $\calD$ contained in $\D E$. Note that each tripod is half of a $\theta$-graph. See Figure \ref{fig2}. By Myers's Theorem we can assume that for each $E\in\CB$ the graph 
$$T(E)=\cup_{D\subset\calD\cap\D E}T(E,D)$$
is truly excellent.

\begin{figure}[h]
        \centering
         \includegraphics[width=8cm]{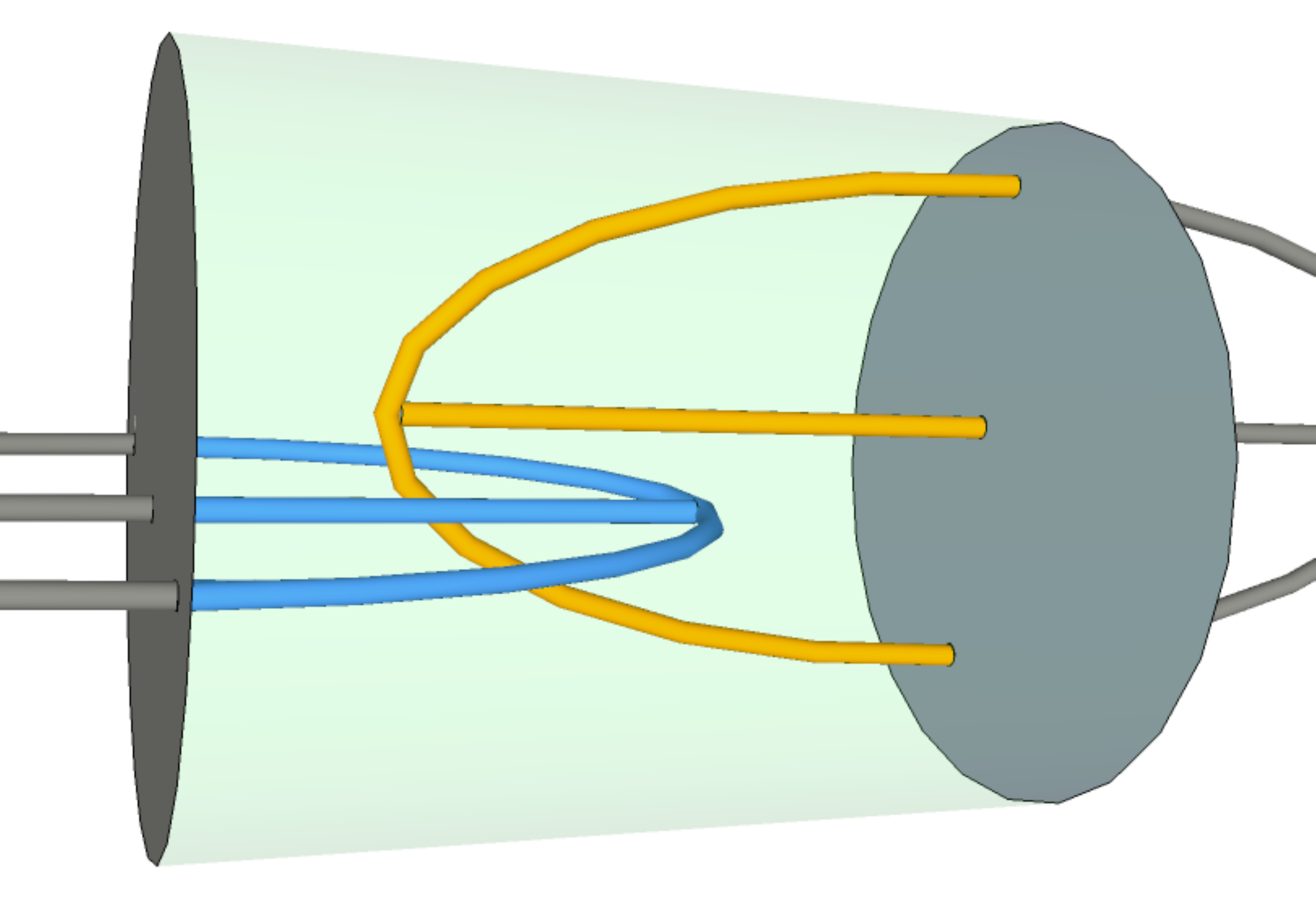}
         \caption{A piece $B_k^j$ and two linked, albeit not excellent, half-$\theta$-graphs.}\label{fig2}
\end{figure}

Notice that $L=\cup_ET(E)$ is a union of $\theta$-graphs, and each of its components is contained in the union of two adjacent $E,F\in\CB$. It remains to prove that $L$ is truly excellent in the handlebody $U$. To see this, notice that $U\setminus L$ is homeomorphic to $W\setminus L$, and that, by construction, cutting $W\setminus L$ along $\calD\setminus L$ determines a decomposition satisfying the assumptions of Lemma \ref{gluing}. This completes the proof of Theorem \ref{sat1}.
\end{proof}

\bigskip

\noindent Department of Mathematics, University of British Columbia.
\newline \noindent
\texttt{jsouto@math.ubc.ca}

\bigskip

\noindent Department of Mathematics, University of Michigan.
\newline \noindent
\texttt{stoverm@umich.edu}

\end{document}